\documentclass[a4paper,12pt]{article}
\usepackage{amsmath,amssymb}
\usepackage{amsthm}
\usepackage[dvipdfmx]{graphicx}
\usepackage{mathrsfs}
\usepackage[active]{srcltx}
\usepackage{amscd}
\usepackage{cases}
\usepackage[noadjust]{cite}
\usepackage{float}
\usepackage{tabularx}
 \newtheorem{theorem}{Theorem}[section]
 \newtheorem{proposition}[theorem]{Proposition}

 \newtheorem{corollary}[theorem]{Corollary}
\theoremstyle{definition}
 \newtheorem{definition}[theorem]{Definition}
 
 \newtheorem{example}[theorem]{Example}
 \newtheorem*{acknowledgements}{Acknowledgements}
\numberwithin{equation}{section}

\newcommand{\R}{\boldsymbol{R}}

\newcommand{\Z}{\boldsymbol{Z}}
\newcommand{\A}{\mathcal{A}}

\renewcommand{\phi}{\varphi}

\renewcommand{\Gamma}{\varGamma}
\newcommand{\ep}{\varepsilon}

\begin{document}
\title{Geometric deformations of cuspidal $S_1$ singularities}
\author{Runa Shimada}
\date{}
\maketitle
\begin{abstract}
To study a deformation of a singularity
taking into
consideration their differential geometric properties,
a form representing the deformation
using only diffeomorphisms on the source space and isometries of the target space
plays a crucial role.
Such a form for an $S_1$ singularity is obtained by the author's previous work.
On this form, we give a necessary and sufficient
condition for such a map is being a frontal.
The form for an $S_1$ singularity with the frontal condition 
can be considered such a form for a cuspidal $S_1$ singularity.
Using this form, 
we investigate geometric properties of cuspidal $S_1$ singularities
and the cuspidal cross caps appearing in the deformation.
\end{abstract}

\section{Introduction}
Singularities are deformed and turn into various singularities.
 The cuspidal $S_1^\pm$ singularities are known as codimension one singularities of frontals.
 The codimension zero singularities of frontals are cuspidal cross caps
  and the codimension one singularities correspond to the appearance/disappearance of codimension zero singularities.  
 The purpose of this paper is to investigate geometric properties of cuspidal $S_1^\pm$ singularities including the deformation.
  The appearance/disappearance of the cuspidal cross caps can be seen in deformations of the cuspidal $S_1^\pm$ singularities. 
An $SO(3)$-{\it normal form} (a {\it normal form} for short)
is a formula for a singular point reducing coefficients
as much as possible by using a diffeomorphism-germ on
the source space and an isometry-germ on the target space.
By the construction, since the coefficients of such a form are geometric invariants of the singular point,
it plays a crucial role to investigate geometric properties of the singularity.
This kind of form is given in \cite{bw,west} for the
Whitney umbrella and is called the Bruce-West normal form,
and there are several papers studying geometry on Whitney umbrella using this form
\cite{bk,diastari,fh-fronts,gsg,hhnsuy,hayashi,hnuy,sym,tari}. 
In \cite{shimada}, extending the notion of the normal form,
a normal form of the $S_1$ singularity including a deformation parameter is given.

In this paper, refining this form, we give a normal form whose $2$-jet is $\A$-equivalent to
$(u,v^2,0)$ or $(u,v^2,uv)$ including deformation parameter (\eqref{eq:normals1} in {Theorem \ref{thm:normals1}).
Moreover, we give a necessary and sufficient condition that \eqref{eq:normals1} is a frontal.
Using this fact, we construct a normal form of the cuspidal $S_k^\pm$ singularities $(k\in\Z_{\geq 0})$
including a deformation parameter ({Corollary \ref{cor:normal}).
Using the form,
we consider the geometric properties of the deformation when $k = 1$.
In the generic deformation of the 
cuspidal $S_1^\pm$ singularities, two cuspidal cross caps merge into a single singularity and then disappear, or vice versa.
To study cuspidal $S_1^\pm$ singularities, we look at geometry on cuspidal
cross caps appearing small deformations of cuspidal $S_1^\pm$
singularities.
We study the asymptotic behavior
of the self-intersection curve appearing in the deformation of cuspidal $S_1$ singularities.
We calculate the bias and secondary cuspidal curvature at the cuspidal cross cap appearing on the deformation of cuspidal $S_1^\pm$ singularities.

The cuspidal $S_k^\pm$ $(k\in\Z_{\geq 0})$ singularities
are map-germs  $\A$-equivalent to the map-germ defined by 
$
(u, v) \mapsto (u, v^2, v^3(u^{k+1}\pm v^2))
$
at the origin. If $k$ is even, then the cuspidal $S_k^+$ singularity
and the cuspidal $S_k^-$ singularity are $\A$-equivalent.
The cuspidal $S_0$ singularity is also called a {\it cuspidal cross cap}.
See  the center figures of Figures \ref{fig:cuspidals1plus}  and \ref{fig:cuspidals1minus}.
Here, two map-germs $f_1:(\R^2,0)\to(\R^3,0)$ and $f_2:(\R^2,0)\to(\R^3,0)$ are {\it $\A$-equivalent} if there exist a coordinate change of source space $\phi$ and a coordinate change of target space $\psi$ such that $f_2=\psi\circ f_1\circ \phi^{-1}$.
A map-germ $f:(\R^2,0)\to (\R^3,0)$ is called a {\it frontal} if there exist 
a unit normal vector field $\nu:(\R^2,0)\to\R^3$ along  $f$ such that $\langle df(X),\nu\rangle=0$ holds  for any $p\in(\R^2,0)$ and $X \in T_p\R^2$.
Let us set $\lambda=\det(f_u, f_v,\nu)$ for a coordinate system $(u,v).$
This $\lambda$ is called the {\it identifier of singularities}, and 
$\lambda$ satisfies $S(f)=\lambda^{-1}(0)$, where $S(f)$ is the set of the singular points of $f$.
A singular point $p\in S(f)$ is said to be {\it non-degenerate} if $d\lambda(p)\ne0$ holds.
If $f:(\R^2,0) \to (\R^3,0)$ is a cuspidal $S_k^\pm$ singularity $(k\in\Z_{\geq 0})$,
then it is a frontal, and the singular point of it is non-degenerate.

\section{Normal forms of cuspidal $S_k^\pm$ singularities including deformations}

To study geometric deformations of
cuspidal $S_1^\pm$ singularities and their
frontality, we give a normal form.
In this paper, we consider $1$-parameter deformation of cuspidal $S_k^\pm$ singularity.
A map-germ $f:(\R^2 \times \R,0) \to (\R^3,0)$ is a {\it deformation of $g:(\R^2,0) \to (\R^3,0)$},  if it is smooth and $f(u, v, 0)=g(u,v)$ and $f(0,0,s)=(0,0,0)$.
In this definition, the parameter $s$ as the third component of the source space is called the {\it deformation parameter}. We define an equivalence relation between two deformations preserving the deformation parameters.
\begin{definition}\label{def:deformeq}
Let $f_1, f_2 : (\R^2 \times \R, 0) \to (\R^3, 0)$ be deformations of $g:(\R^2,0)\to(\R^3,0)$. Then $f_1$ and
$f_2$ are {\it equivalent as deformations} if there exist orientation preserving diffeomorphism-germs $\phi:(\R^2 \times \R,0) \to (\R^2 \times \R,0)$
with the form
\begin{eqnarray}
\phi(u,v,s)=(\phi_1(u,v,s),\phi_2(u,v,s),\phi_3(s))\quad
\left(
\dfrac{d\phi_3}{ds}(0)>0\right)
\label{eq:phi}
\end{eqnarray}
and 
$\psi:(\R^3 ,0) \to (\R^3,0)$
such that
$\psi \circ f_1 \circ \phi^{-1}(u,v,s)=f_2(u,v,s)$ holds.
\end{definition}
See (\cite[Section 2]{shimada}) for detail.

Let ${f_s}^\pm$ be germs defined by
$${f_s}^\pm:(\R^2 \times \R,0) \ni (u,v,s) \mapsto (u,v^2,v^3(u^2\pm v^2)+sv^3) \in (\R^3,0).$$
Then ${f_s}^+$ $({f_s}^-)$ is a deformation of a cuspidal 
$S_1^+$ $(S_1^-) $ singularity which is
a typical deformation of the cuspidal $S_1^+$ $(S_1^-) $ singularity.
We can observe two cuspidal cross caps appear in the deformation.
See Figure \ref{fig:cuspidals1plus} and Figure \ref{fig:cuspidals1minus}.
\begin{figure}[h!]
\centering
\includegraphics[width=70mm]{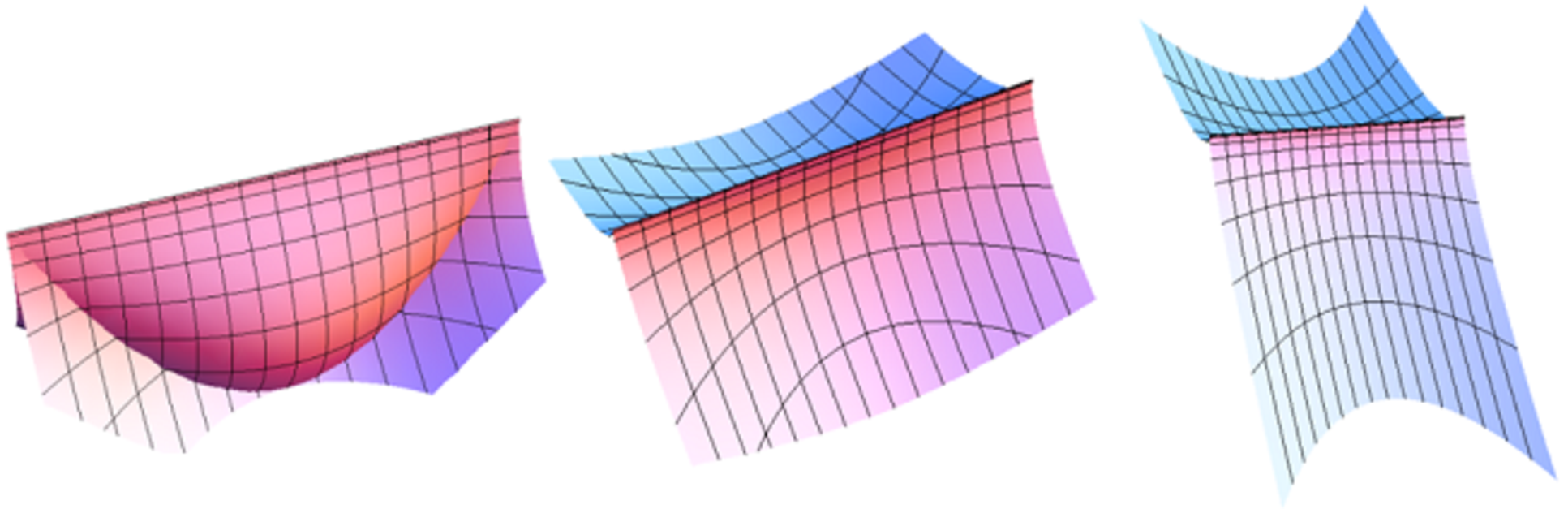}
\caption{Deformation of cuspidal $S_1^+$ singularity $($from left to right $f_{-1}^+, f_{0}^+$ and $f_{1}^+)$}
\label{fig:cuspidals1plus}
\end{figure}

\begin{figure}[htbp]
\centering
\includegraphics[width=70mm]{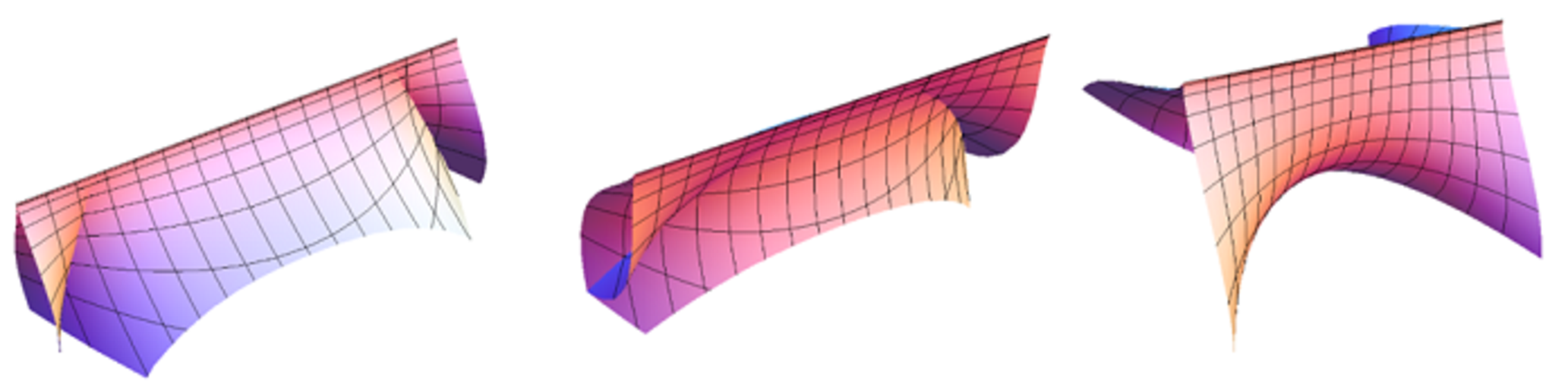}
\caption{Deformation of cuspidal $S_1^-$ singularity $($from left to right $f_{-1}^-,  f_{0}^-, $ and $f_{1}^-)$}
\label{fig:cuspidals1minus}
\end{figure}

\begin{theorem}{\rm(\cite[Theorem 2.3]{shimada})}\label{thm:normals1}
Let $ f : (\R^2 \times \R, 0) \to (\R^3, 0)$ be a deformation of 
$g:(\R^2, 0) \to (\R^3, 0)$ such that
the $2$-jet of $g$ is $\A$-equivalent to $(u,v^2,0)$ or $(u,v^2,uv)$.
Then there exist an orientation preserving diffeomorphism-germ $\phi : (\R^2 \times \R, 0) \to (\R^2 \times \R, 0)$ with the form \eqref{eq:phi},
$T \in SO(3)$ and the functions $f_{21}, f_{31} \in C^\infty(1,1), f_{24}, f_{33}, f_{34} \in C^\infty(2,1), f_{32} \in C^\infty(3,1)$ such that
\begin{eqnarray}
{f_{{\rm n1}}}^s&=&
T \circ f \circ \varphi(u,v,s)\nonumber\\ 
&=&(u,u^2f_{21}(u)+v^2+usf_{24}(u,s),\label{eq:normals1}\\
& &\hspace{30mm}
u^2f_{31}(u)+v^2f_{32}(u,v,s)+vf_{33}(u,s)+usf_{34}(u,s)),\nonumber
\end{eqnarray}\nonumber
where
$f_{32}(0,0,0)=f_{33}(0,0)=0$.
If the $2$-jet of $g$ is $\A$-equivalent to $(u,v^2,0)$, then $(f_{33})_u(0,0)=0 $ holds, and if it is $\A$-equivalent to $(u,v^2,uv)$, then
$(f_{33})_u(0,0)\ne0$ holds.
\end{theorem}

See \cite[Theorem 2.3]{shimada} for the proof of the case that the 
$2$-jet of $g$ is $\A$-equivalent to $(u,v^2,0)$.
The proof also works for the case that the $2$-jet of $g$ is $\A$-equivalent to $(u,v^2,0)$ or $(u,v^2,uv)$.
In Theorem \ref{thm:normals1}, the given $f$ and ${f_{{\rm n1}}}^s$ are equivalent as deformations (Definition \ref{def:deformeq}), and they
have the same differential geometric properties.
The uniqueness of the normal form holds, see \cite[Proposition 2.5]{shimada},
and the proof also work for that the case of the $2$-jet of $g$
is $\A$-equivalent to $(u,v^2,0)$ or $(u,v^2,uv)$.
For the frontality of the form  ${f_{{\rm n1}}}^s$, we have the following theorem.
 \begin{theorem}\label{thm:fr}
 A map-germ ${f_{{\rm n1}}}^s$ is frontal for any $s$ if and only if $f_{33}(u,s)=0$ holds identically.
 \end{theorem}
 \begin{proof}
 We assume that $f_{33}(u,s)\equiv0$,where $\equiv$ implies that the equality holds identically. By a direct calculation, one can obtain a unit normal vector $\nu$ such that $({f_{{\rm n1}}}^s)_u\cdot\nu=({f_{{\rm n1}}}^s)_v\cdot\nu=0$.
 On the other hand, we assume that ${f_{{\rm n1}}}^s$ is a frontal.
 We set a unit normal vector of ${f_{{\rm n1}}}^s$ to be $\nu=(\nu_1,\nu_2,\nu_3)$.
 By the  definition of frontality it hold that
\begin{eqnarray}
({f_{{\rm n1}}}^s)_v\cdot\nu(u,0,s)=f_{33}(u,s)\nu_3(u,0,s)=0.\label{thm:frontal}
\end{eqnarray}
We assume $\nu_3(0,0,0)=0$.
Then since $f_u(0,0,0)=(1,0,0)$, it hold that $\nu(0,0,0)=(0,1,0)$.
Thus $({f_{{\rm n1}}}^s)_{vv}\cdot\nu(0,0,0)=2.$
Here, $({f_{{\rm n1}}}^s)_{vv}\cdot\nu(0,0,0)$ can rewrite $-({f_{{\rm n1}}}^s)_v\cdot\nu_v(0,0,0)$, however since
 $\operatorname{Ker}d{f_{{\rm n1}}}^s =\langle \partial_v\rangle$, it holds that $({f_{{\rm n1}}}^s)_v(0,0,0)=0$.
 This is a contradiction.
 Therefore $\nu_3(0,0,0)\ne0$, and this implies $f_{33}(u,s)=0$ holds identically by \eqref{thm:frontal}.
\end{proof}

\begin{corollary}\label{cor:normal}
Let $f:(\R^2 \times \R,0) \to (\R^3,0)$ be a deformation of $g:(\R^2,0)\to(\R^3,0)$ such that the $2$-jet of $g$ is $\A$-equivalent to $(u,v^2,0)$ or $(u,v^2,uv)$ and $f$  is frontal for any $s$.
Then there exist an orientation preserving diffeomorphism-germ $\phi : (\R^2 \times \R, 0) \to (\R^2 \times \R, 0)$ with the form \eqref{eq:phi}, 
$T \in SO(3)$ and the functions $f_{21}, f_{31} \in C^\infty(1,1), f_{24}, f_{34} \in C^\infty(2,1), f_{32} \in C^\infty(3,1)$ such that
\begin{eqnarray}
{f_{{\rm n2}}}^s&=&
T \circ f \circ \varphi(u,v,s)\nonumber\\ 
&=&(u,u^2f_{21}(u)+v^2+usf_{24}(u,s),\label{eq:normal}\\
& &\hspace{30mm}
u^2f_{31}(u)+v^2f_{32}(u,v,s)+usf_{34}(u,s)),\nonumber
\end{eqnarray}\nonumber
where
\begin{eqnarray}
f_{32}(u,v,s)=c_{0}(u,s)+vc_{1}(u,s)+v^2c_{2}(u,v^2,s)+v^3c_{3}(u,v^2,s)\label{f32}
\end{eqnarray}
 and
 $f_{32}(0,0,0)=c_{0}(0,0)=0$.
Furthermore, $f(u,v,0)=g(u,v)$ is a cuspidal $S_k^+$ singularity $($respectively, cuspidal $S_k^-$ singularity$)$  if and only if  
$\partial^{i} c_{1}/\partial u^{i}(0,0,0)=0$ $(i=1,\ldots,k)$ and
 $\partial^{k+1}c_{1} /\partial u^{k+1}(0,0,0) c_{3}(0,0,0) >0$ $($respectively, $<0)$ hold $(k\in\Z_{\geq 0})$.
 If $(dc_{1}/ds)(0,0) \ne 0$, then one can further reduce $c_{1}(0,s)=s$.
 \end{corollary}

\begin{proof}
By Theorem \ref{thm:normals1}, substituting $f_{33}(u,s)=0$ into 
\eqref{eq:normals1}, we have the assertion.
 Since $f(u,v,0)$ at $(0,0)$ is neither a cuspidal edge nor a cuspidal cross cap, $c_{1}(0,0)=0$ and $(c_{1})_{u}(0,0)=0$ hold.
By \cite[Theorem 3.2]{saji}, we see that ${f_{{\rm n2}}}^s$ of the form
\eqref{eq:normal} with these conditions is
a cuspidal $S_k^+$ singularity $($respectively, cuspidal $S_k^-$ singularity$)$ if and only if
$\partial^{i} c_{1}/\partial u^{i}(0,0,0)=0$ $(i=1,\ldots,k)$ and
 $\partial^{k+1}c_{1} /\partial u^{k+1}(0,0,0) c_{3}(0,0,0) >0$ $($respectively, $<0)$ hold $(k\in\Z_{\geq 0})$.
If $(dc_{1}/ds)(0,0) \ne 0$, 
then one can assume $c_{1}(0,s)=s$ by a change of the deformation parameter. 
\end{proof}

We remark that normal forms for the cuspidal $S_k^\pm$ singularities
itself is given in \cite{os}.
The uniqueness of the form \eqref{eq:normal} holds 
as in the remark just after Theorem \ref{thm:normals1}. 
If
$(dc_{1}/ds)(0,0) \ne 0$,
then ${f_{{\rm n2}}}^s$ is a generic deformation.
In what follows, we assume $(dc_{1}/ds)(0,0) \ne 0$ and
$c_{1}(0,s)=s$ in ${f_{{\rm n2}}}^s$.
When $k=1$, the form ${f_{{\rm n2}}}^s$ is called the {\it normal form of the deformations of a cuspidal $S_1^\pm$ singularity}.
We see the set of singular points $S_1({f_{{\rm n2}}}^s)$ of ${f_{{\rm n2}}}^s$ is
$$S_1({f_{{\rm n2}}}^s)=\{(u,v)\,|\, v=0\}.$$
We set
$S_2({f_{{\rm n2}}}^s)$
to be the set of singular points that are not cuspidal edge.
Then it holds that
$$S_2({f_{{\rm n2}}}^s)=\{(u,v)\,|\, v=0, c_{1}(u,s)=0\}.$$

\section{Minimal frontalization}
Let $C^{\infty}(2,3)_f$ be the set of $C^{\infty}$-map-germs which to be a frontal.
A map $C^{\infty}(2,3)\to C^{\infty}(2,3)_f$ is called a {\it frontalization}.
See \cite[Definition 3.5]{mno} for an example
of such a map. 
Let $C^{\infty}(2,3)_a$ be the set of $C^{\infty}$-map-germs
satisfying the $2$-jet is $\A$-equivalent to $(u,v^2,0)$ or $(u,v^2,uv)$,
and let $C^{\infty}(2,3)_b$ be the subset of $C^{\infty}(2,3)_a$ which
consists of map-germs to be frontals.
Taking $f\in C^{\infty}(2,3)_a$,
there exists an orientation preserving diffeomorphism-germ $\phi : (\R^2, 0) \to (\R^2,0)$, $T \in SO(3)$
 and functions $f_{21}, f_{31}, f_{33} \in C^\infty(1,1)$ and $f_{32} \in C^\infty(2,1)$ such that  
 \begin{eqnarray}
T \circ f \circ \varphi(u,v)=(u,u^2f_{21}(u)+v^2,
u^2f_{31}(u)+v^2f_{32}(u,v)+vf_{33}(u)).\label{frs0}
\end{eqnarray}
If the $2$-jet of $f$ is $\A$-equivalent to
$(u,v^2,0)$, then $(f_{33})_u(0)=0$ and
the $2$-jet of $f$ is $\A$-equivalent to
$(u,v^2,uv)$, then $(f_{33})_u(0)\ne0$ hold.
Moreover, the uniqueness of the normal form \eqref{frs0} holds.
In fact, substituting $s=0$ into the proof of \cite[Proposition2.5]{shimada},
the proof turns to the proof for the uniqueness of \eqref{frs0}.
By Theorem \ref{thm:fr}
 the right-hand side of  \eqref{frs0} is a frontal if and only if
$f_{33}(u)=0$. Thus, 
the function corresponding to $vf_{33}(u)$ is
called the {\it obstruction for frontality of} $f$.

Then $f$ can be divided as
$$
T \circ f \circ \varphi(u,v)=f_1(u,v)+(0,0,vf_{33}(u)),
$$
and
the map corresponding to $f_1(u,v)$ is
called the {\it frontal part} of $f$.
An obstruction for frontality is studied for a map written in the
form $(u,f_2(u,v),f_3(u,v))$ in \cite{os}. In our case, by
the uniqueness of the normal form \eqref{eq:normals1} 
(\cite[Proposition 2.5]{shimada}),
the obstruction for frontality and the frontal part are
well-defined, and is uniquely
determined from the given $f$.
We define $F(f)$ is the frontal part of $f$.
Then we obtain a map $F:C^{\infty}(2,3)_a\to C^{\infty}(2,3)_b$ and
it holds that $F(f)\in C^{\infty}(2,3)_b$.
The map $F$ is a surjection.
Moreover,  $F|_{C^{\infty}(2,3)_b}$ is the identity map, 
and does not change the frontal part,
we call that $F$ is the {\it minimal frontalization}.
See the center figures of
Figures \ref{fig:defof} and \ref{fig:frontalization} for this frontalization.
In \cite{mond}, a classification of $2$-jet is given.
Table 1 shows the minimal  frontalization of the map-germs satisfying the $2$-jet is $\A$-equivalent to $(u,v^2,0)$ or $(u,v^2,uv)$.
\begin{table}[h]\label{tab:mond}
    \centering
    \begin{tabular}{cccc}
        \hline
        Name &Germ & obstruction for frontality  & frontal part  \\ \hline 
        $S_0$ & $(u,v^2,uv)$ & $uv$ & $(u,v^2,0)$\\ \hline
       $S_k$ & $(u,v^2,v^3\pm u^{k+1}v)$ & $u^{k+1}v$ & $(u,v^2,v^3)$ \\ \hline
        $B_k$& $(u,v^2,u^2v\pm v^{2k+1})$ & $u^2v$& $(u,v^2,\pm v^{2k+1})$ \\ \hline
        $C_k$ &  $(u,v^2,uv^3\pm u^kv)$ & $u^kv$ & $(u,v^2,uv^3)$ \\ \hline
         $F_4$ &  $(u,v^2,u^3v+v^5)$ & $u^3v$ & $(u,v^2,v^5)$ \\ \hline
    \end{tabular}
     \caption{Minimally frontalized germ from Mond's classification}
\end{table}

Let $C^\infty((2,1),3)_a$ be the set of deformations of 
a map in $C^\infty(2,3)_a$ defined in Definition \ref{def:deformeq}}, and let 
$C^\infty((2,1),3)_b$ be the subset of 
$C^\infty((2,1),3)_a$ which consists of deformations
to be frontals for any $s$ and the $2$-jet of $s=0$ case is $\A$-equivalent to $(u,v^2,0)$.
Taking a deformation $f \in C^\infty((2,1),3)_a$,
there exist an orientation preserving diffeomorphism-germ $\phi : (\R^2, 0) \to (\R^2,0)$, $T \in SO(3)$ and functions $f_{21}, f_{31} \in C^\infty(1,1), f_{24}, f_{33}, f_{34} \in C^\infty(2,1), f_{32} \in C^\infty(3,1)$  
 such that \eqref{eq:normals1} holds.
By Theorem \ref{thm:fr}, the right-hand side of \eqref{eq:normals1}
is a frontal
for any $s$ if and only if $f_{33}(u,s)=0$.
Thus the function $vf_{33}(u,s)$ is called the
{\it obstruction for frontality of the deformation}, and
we define {\it frontal part} in the same way as above.
We define $F(f)$ is the frontal part of $f$.
Similarly we obtain a  map 
$F:C^\infty((2,1),3)_a\to C^\infty((2,1),3)_b$ and
by Theorem \ref{thm:fr}, $F(f)\in C^{\infty}((2,1),3)_b$ hold.
Moreover,
the map  $F$ is a surjection and $F|_{C^{\infty}((2,1),3)_b}$ is the identity map,
we call that $F$ is the {\it minimal frontalization of a deformation} as well.
When $s>0$, 
map-germs $f \in C^\infty((2,1),3)_a$ do not have singular point,
it is a frontal.
However, as a deformation,
the surfaces appearing in the family of minimally frontalized deformation
have singular points for any $s$.
Table 2 shows the minimal  frontalization of a deforrmation  of the map-germs satisfying the $2$-jet of $s=0$ case is $\A$-equivalent to $(u,v^2,0)$.

\begin{table}[h]\label{tab:mond}
    \centering
    \begin{tabular}{cccc}
        \hline
        Name &Germ & \shortstack{\\obstruction \\ for frontality} &  frontal part  \\ \hline 
       $S_k$ & $(u,v^2,v^3\pm u^{k+1}v+vs+v^3s)$ & $u^{k+1}v +vs$ & $(u,v^2,v^3+v^3s)$ \\ \hline
        $B_k$& $(u,v^2,u^2v\pm v^{2k+1}+vs+v^3s)$ & $u^2v+vs$& $(u,v^2,\pm v^{2k+1}+v^3s)$ \\ \hline
        $C_k$ &  $(u,v^2,uv^3\pm u^kv+vs+v^3s)$ & $u^kv+vs$ & $(u,v^2,uv^3+v^3s)$ \\ \hline
         $F_4$ &  $(u,v^2,u^3v+v^5v+vs+v^3s)$ & $u^3v+vs$ & $(u,v^2,v^5+v^3s)$ \\ \hline
    \end{tabular}
     \caption{Example of minimally frontalized germs of deformations from Mond's classification}
\end{table}

\begin{example}\label{ex:ex2}
Let $ f : (\R^2 \times \R, 0) \to (\R^3, 0)$ be a deformation of 
$g:(\R^2, 0) \to (\R^3, 0)$ such that
the $2$-jet of $g$ is $\A$-equivalent to $(u,v^2,0)$ defined by
$$f(u,v,s)=(u,u^2+v^2,u^2+v^3s+u^2v^3+v^5+v^7+vs+u^2v).$$
\end{example}
\begin{figure}[htbp]
\centering
\includegraphics[width=60mm]{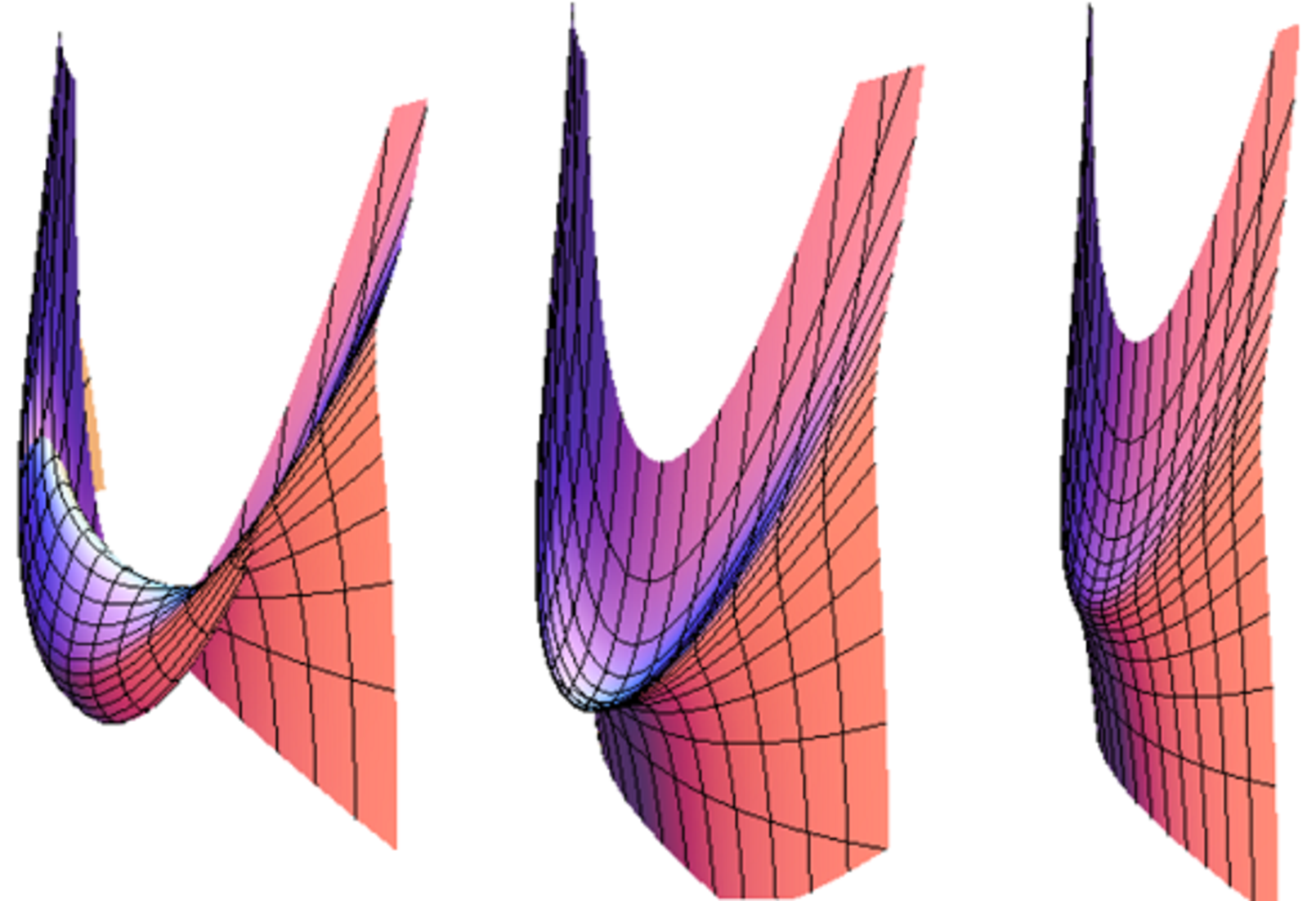}
\caption{The surfaces in Example \ref{ex:ex2} (from left to right $s=-1, 0, 1)$}
\label{fig:defof}
\end{figure}
The obstruction for frontality of $f$ in Example \ref{ex:ex2} is $vs+u^2v$.
Therefore if this term vanishes, it is minimal frontalization of $f$.
\begin{figure}[htbp]
\centering
\includegraphics[width=60mm]{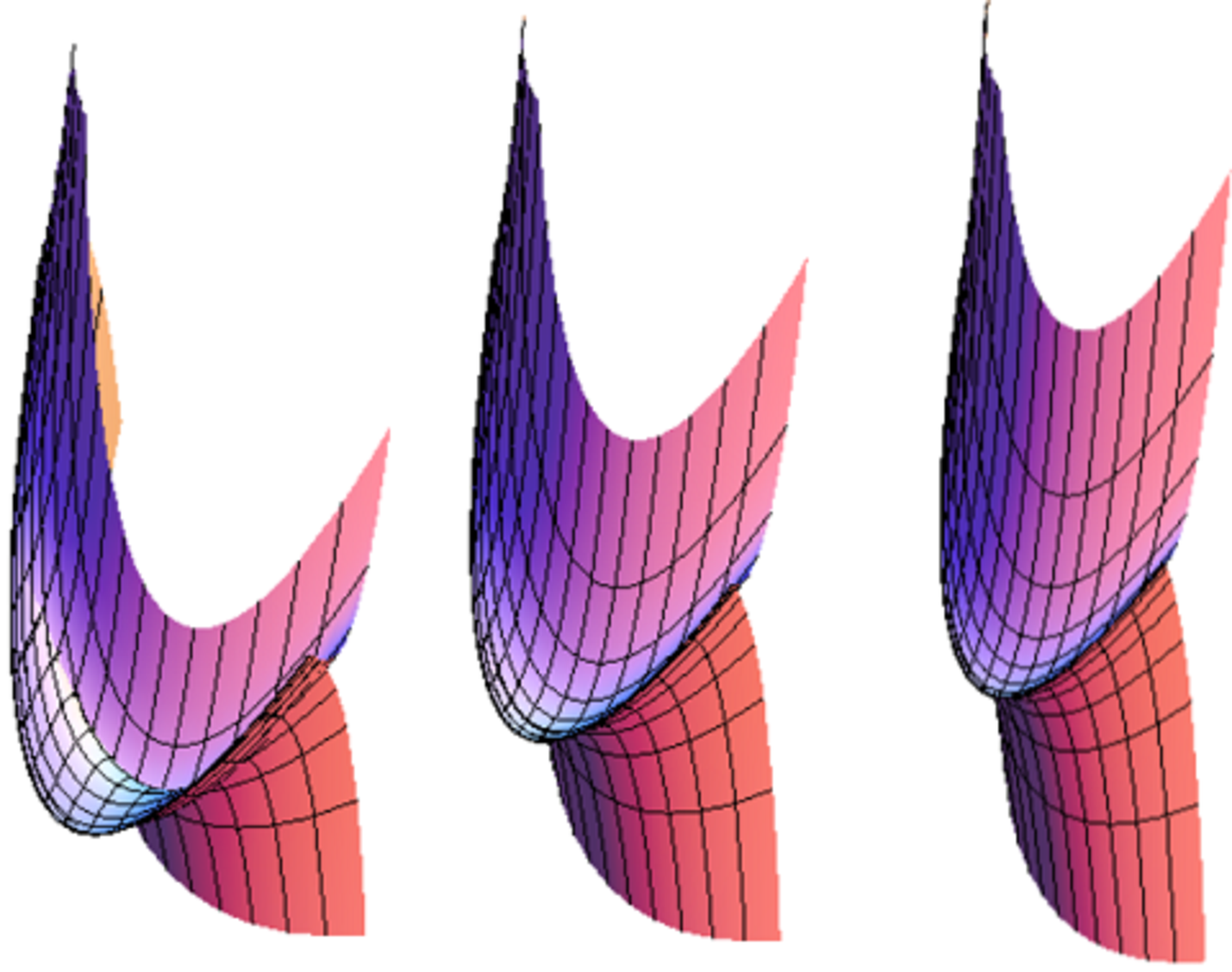}
\caption{Minimal frontalization of  the map in Example \ref{ex:ex2}  (from left to right $s=-1, 0, 1)$}
\label{fig:frontalization}
\end{figure}

\section{Geometry on deformations of cuspidal $S_1$ singularities}
In this section, we consider geometry on the case of $k=1$  in Corollary \ref{cor:normal}.
Firstly, we investigate the location of the singular point, and
we study differential geometric properties of 
deformations of cuspidal $S_1^\pm$ singularities.

We set $f={f_{{\rm n2}}}^s$ (see \eqref{eq:normal}). 
Here, we assume $c_{1}(0,s)=s$ 
(See the remark just after the proof of Corollary \ref{cor:normal}).

\subsection{Description of singular point}

 To obtain the location of the singular point, we set $s=-\tilde{s}^2$,
since $S_2(f)  \neq \emptyset$ is equivalent to $s\le 0$.
If $(u, v)\in S_2(f)$, then $v = 0$ 
and $u$ depends on $\tilde s$.
We set this function $u(\tilde s)$.
Since $c_{1}(0,0)=(c_{1})_u(0,0)=0$ in \eqref{eq:normal}, we set
\begin{eqnarray}
c_{1}(u,s)=s+usd_{1}(s)+u^2d_{2}(s)+u^3d_{3}(s)+u^4d_{4}(u,s).\label{c1}
\end{eqnarray}
Rewriting $f_{32}$ as \eqref{f32} and \eqref{c1},
then we have the following theorem.
\begin{theorem}\label{thm:singcurve}
If $(u,v) \in S_2(f)$ and $d_{2}(0)>0$, then $v=0$ holds, and 
the function $u(\tilde s)$ can be expanded as follows\/{\rm :}
\begin{align}
&u(\tilde s)
=
\dfrac{1}{d_{20}}\tilde{s}+\dfrac{1}{2d_{20}^4}\Bigl(d_{1}(0)d_{20}^2-d_{3}(0)\Bigr)\tilde{s}^2\nonumber\\
&+\dfrac{1}{8d_{20}^7}\Bigl(\bigl(d_{1}^2(0)+4(d_{2})_{s}(0)\bigr)d_{20}^4\label{eq:u}\\
&-2\bigr(3d_{1}(0)d_{3}(0)+2d_{4}(0,0)\bigr)d_{20}^2+5d_{3}^2(0)\Bigr)\tilde s^3+O(4),
\nonumber
\end{align}
where $d_2(0) = d_{20}^2.$
If $u(\tilde s)\ne 0$, $f$ at $(u(\pm\tilde s),0)$ are both the cuspidal cross caps.
\end{theorem}
\begin{proof}
Let us set
$$
u=\alpha_1 \tilde s+\alpha_2 \tilde s^2+\alpha_3 \tilde s^3+\alpha_4(\tilde s)\tilde s^4.
$$
Differentiating $c_{1}(\alpha_1 \tilde s+\alpha_2 \tilde s^2+\alpha_3 \tilde s^3+\alpha_4(\tilde s)\tilde s^4,-\tilde s^2) =0$ three times, we have
$$
\alpha_1=\dfrac{1}{d_{20}},\quad \alpha_2=\dfrac{1}{2d_{20}^4}\Bigl(d_{1}(0)d_{20}^2-d_{3}(0)\Bigr),\\
$$

$\alpha_3=\dfrac{1}{8d_{20}^7}\Bigl(\bigl(d_{1}^2(0)+4(d_{2})_{s}(0)\bigr)d_{20}^4
-2\bigr(3d_{3}(0)d_{1}(0)+2d_{4}(0,0)\bigr)d_{20}^2+5d_{3}^2(0)\Bigr).$

Since 
$\operatorname{Ker}df= \langle \partial_v\rangle$
at any $(u(\tilde s),0)$,
by  \cite[Theorem 1.4]{fsuy}, we obtain the assertion that
$f$ at $(u(\pm \tilde s),0)$ are both the cuspidal cross caps.
\end{proof}

If $d_{2}(0)<0$, the same calculation can be done by setting $d_{2}(0) = -d_{20}^2$,
and we obtain the same results.

\subsection{Self-intersection curves}
In this section we focus on self-intersection curves of deformations of cuspidal $S_1$ singularities.
In particular, deformations of the geodesic curvature and the normal curvature of self-intersection curves are considered.

The self-intersection curves are determined by $f(u_1,v_1)=f(u_2,v_2)$ in general.
In our case,
by looking the first component of ${f_{{\rm n2}}}^s$, we see $u_1=u_2$, we rewrite as $u$.
Focusing on the second component, $v_1^2=v_2^2$ holds.
If we set $v=v_1=-v_2$, then
the self-intersection curves are determined by
$i(u,v,s)=c_{1}(u,s)+v^2c_{3}(u,v^2,s)=0$ from the third component.
Since $i_u(u,0,s)=sd_1(s)+2ud_2(s)+3u^2d_3(s)+4u^3d_4(u,s)+u^4(d_4)_u(u,s)$ and
by the assumption $d_2(0)>0$, it holds that $i_u(u,0,s)\ne0$ for
small $s\ne0$ and $u\ne0$.
By the implicit function theorem, there exists a function $u(v,s)$ such that $i(u(v,s),v,s)=c_{1}(u(v,s),s)+v^2c_{3}(u(v,s),v^2,s)=0$ holds
for any $v$ and $s\ne0$.

\begin{theorem}
Let  $\kappa_g(v,\tilde s)$ be the geodesic curvature and $\kappa_n(v,\tilde s)$ be the normal curvature of self-intersection curves $v\mapsto f(u(v,\tilde s),v,-\tilde s^2)$.
Then these curvatures can be
expanded as follows
\begin{eqnarray}
\kappa_g(0,\tilde s)&=&
\pm\dfrac{(-2f_{21}(0)c_3(0,0,0)+(c_{1})_{uu}(0,0))}{c_3(0,0,0)}+O_{\tilde s}(1)
\label{kgbunshi}\\
\kappa_n(0,\tilde s)&=&
2f_{31}(0)+O_{\tilde s}(1)\label{knbunshi}
\end{eqnarray}
In particular, both of the geodesic curvature and 
the normal curvature are bounded for fixed $\tilde s\ne0$.
\end{theorem}
\begin{proof}
We set that $\hat c(u(v,s),v,s)={f_{{\rm n2}}}^s(u(v,s),v,s).$
By  calculating the partial derivatives of $i(u(v,s),v,s)=0$, we have
\begin{eqnarray}
u_v&=&u_{vvv}=0,\label{uv}\\
u_{vv}&=&-\dfrac{2c_{3}}{(c_{1})_{u}},\label{uvv}\\
u_{vvvv}&=&-\dfrac{12}{(c_{1})_{u}^3}  \Bigl(2 (c_{1})_{u}^2 (c_{3})_{v}
-2 (c_{1})_{u}(c_{3})_{u}{c_{3}}
+(c_{1})_{uu}c_{3}^2\Bigr),\label{uvvvv}
\end{eqnarray}
where the functions are evaluated at $v=0$ and $s=-\tilde s^2$.
By a calculation of $\hat c$ using \eqref{uv}, we have
\begin{eqnarray}
\hat c_{vv}&=&{(f_{{\rm n2}}}^s)_{vv}
+u_{vv}{(f_{{\rm n2}}}^s)_u,\label{cvv}\\
\hat c_{vvvv}&=&{(f_{{\rm n2}}}^s)_{vvvv}
+u_{vvvv}{(f_{{\rm n2}}}^s)_u
+6u_{vv}{(f_{{\rm n2}}}^s)_{uvv}\label{cv4}
+3u_{vv}^2{(f_{{\rm n2}}}^s)_{uu},
\end{eqnarray}
where the functions are evaluated at $v=0$ and $s=-\tilde s^2$.
Here, $({f_{{\rm n2}}}^s)_{vv}(u,0,s)=(0,2,2c_0(u,s))$ and 
$({f_{{\rm n2}}}^s)_{vvvv}(u,0,s)=(0,0,24c_2(u,s))$.

Self-intersection curves are even functions with respect to $v$.
Thus the geodesic curvature $\lim_{v\to0} \kappa_g(v,\tilde s)$ and the normal curvature $\lim_{v\to0}\kappa_n(v,\tilde s)$ for fixed $\tilde s\ne0$ are calculated as in Proposition \ref{prop:kgkn}. By a direct calculation, we have
\begin{eqnarray}
\nu(u,0,s)&=&
(-4u f_{31}(u) + 4 u f_{21}(u)c_{0}(u, s) + 2 s f_{24}(u, s) c_{0}(u, s) - 2 s f_{34}(u, s)\nonumber\\
& &+ 2 u^2 c_{0}(u, s) f_{21}'(u) 
- 2 u^2 f_{31}'(u) + 2usc_{0}(u, s)(f_{24})_u(u,s) -2us(f_{34})_u(u,s),\nonumber\\
& &-2c_{0}(u,s),2)/\delta\nonumber \\
&=&
(0,0,1)+O_{\tilde s}(1),\label{nu}
\end{eqnarray}
and $'=d/du$ and $\delta$ stands for the length of the numerator of it.
Using \eqref{uvv}, \eqref{uvvvv}, \eqref{cvv}, \eqref{cv4} and \eqref{nu}, we have
\begin{eqnarray}
& &\det(\hat c_{vv}, \hat c_{vvvv}, \nu)\Big\vert_{v=0}\nonumber\\
&=&\det\Bigl({(f_{{\rm n2}}}^s)_{vv}+u_{vv}{(f_{{\rm n2}}}^s)_u,{(f_{{\rm n2}}}^s)_{vvvv}
+u_{vvvv}{(f_{{\rm n2}}}^s)_u\nonumber\\
& &\hspace{60mm}+6u_{vv}{(f_{{\rm n2}}}^s)_{uvv}\label{cvvvv}
+3u_{vv}^2{(f_{{\rm n2}}}^s)_{uu},\nu\Bigr)\Big\vert_{v=0}\nonumber\\
&=&\dfrac{1}{(c_{1})_{u}^3}\Bigl(\det\bigl({(f_{{\rm n2}}}^s)_{vv},-12A{(f_{{\rm n2}}}^s)_{u}+B(c_1)_u,\nu\bigr)\nonumber\\
& &\hspace{60mm}-2c_3\det\bigl({(f_{{\rm n2}}}^s)_{u},12Cc_3^2({f_{{\rm n2}}}^s)_{uu},(c_1)_u ,\nu\bigr)\Bigr)\Big\vert_{v=0}\nonumber\\
&=&\dfrac{24c_3(0,0,0)^2}{(c_{1})_u(0,0)^3}\Bigl(-2f_{21}(0)c_3(0,0,0)+(c_{1})_{uu}(0,0)^2\Bigr)+O_{\tilde s}(1),\label{kgbunshi}\\
& &|\hat c_{vv}|\Big\vert_{v=0}\nonumber\\
&=&\dfrac{1}{|(c_{1})_u|}|(c_{1})_u{(f_{{\rm n2}}}^s)_{vv}-2c_3{(f_{{\rm n2}}}^s)_{u}|\Big\vert_{v=0}\nonumber\\
&=&\left|\dfrac{2c_{3}(0,0,0)}{(c_{1})_u(0,0)}\right|+ O_{\tilde s}(1),\label{kgbunbo}\\
& &\hat c_{vvvv}\cdot\nu\Big\vert_{v=0}\nonumber\\
 &=&\dfrac{1}{(c_1)_u}\Bigl({(f_{{\rm n2}}}^s)_{vvvv}(c_1)_{u}^2-12(c_1)_{u}c_3{(f_{{\rm n2}}}^s)_{uvv}+12c_3^2{(f_{{\rm n2}}}^s)_{uu}\Bigr)\cdot \nu\Big\vert_{v=0}\nonumber\\
&=& \dfrac{48f_{31}(0)c_{3}(0,0,0)^2}{(c_{1})_u(0,0)^2}+O_{\tilde s}(1),\label{knbunshi}
      \end{eqnarray}
     where $A, B, C$ are some functions which do not used in the later calculations.
Here, $O_{\tilde s}(n)$ stands for the terms whose degrees are equal to or greater than $n$ for $\tilde s$.
By Proposition \ref{prop:kgkn}, we obtain the assertion.
\end{proof}

When $s=0$, only cuspidal $S_1^-$ singularity has self-intersection curves.
Therefore, we assume $(c_1)_{uu}(0,0)c_3(0,0,0)<0$.

\begin{theorem}
If $(c_1)_{uu}(0,0)c_3(0,0,0)<0$, then 
the geodesic curvature $\kappa_g$ and the normal curvature 
$\kappa_\nu$ of
the each branch of self-intersection curves
are bounded and they
are 
$$
\kappa_g=
\dfrac{2f_{21}(0)c_{3}(0,0,0)-(c_{1})_{uu}(0,0)}{c_{3}(0,0,0)},\quad
\kappa_\nu=
2f_{31}(0).
$$
Furthermore, these coincide with the absolute value of the limits of
the geodesic curvature and the limits of the normal curvature of
the self-intersection curves on the case of $\tilde s\ne0$
respectively.
\end{theorem}
\begin{proof}
By \eqref{kgbunshi} and \eqref{kgbunbo}, limits of the geodesic curvature is
$$\lim_{\tilde s\to0}\lim_{v\to0}\kappa_g(v,\tilde s)=\dfrac{2f_{21}(0)c_{3}(0,0,0)-(c_{1})_{uu}(0,0)}{c_{3}(0,0,0)},$$
and using \eqref{kgbunbo} and \eqref{knbunshi}, limits of the normal curvature is
$$\lim_{\tilde s\to0}\lim_{v\to0}\kappa_n(v,\tilde s)=2f_{31}(0).$$
Since the self-intersection curves of ${f_{{\rm n2}}}^s(u,v,0)$
are determined by $i(u,v,s)=c_{1}(u,s)+v^2c_{3}(u,v^2,s)=0$ with $s=0$,
since $c_3(0)\ne0$, each branch of curves of
$i^{-1}(0)$ is not tangent to $\partial_v$.
Therefore there exists a function $v(u)$ such that $i(u,v(u))=0$.
By  calculating the partial derivatives of $i(u,v(u))=0$, we have
$$(v')^2=-\dfrac{(c_1)_{uu}}{2c_3}.$$
We set $c(u,v(u))={f_{{\rm n2}}}^s(u,v(u))$. Then the geodesic curvature is
\begin{eqnarray}
\kappa_g(0)=\dfrac{|c',c'',\nu|}{|c'|^3}&=&
\dfrac{\det(({f_{{\rm n2}}}^s)_{u},({f_{{\rm n2}}}^s)_{uu}+({f_{{\rm n2}}}^s)_{vv}(v')^2,\nu)}{|({f_{{\rm n2}}}^s)_{u}|^3}\nonumber\\
&=&
\dfrac{2f_{21}(0)c_{3}(0,0,0)-(c_{1})_{uu}(0,0)}{c_{3}(0,0,0)},\nonumber
\end{eqnarray}
and the normal curvature is
$$\kappa_n(0)=\dfrac{c''\cdot\nu}{|c'|^2}=(({f_{{\rm n2}}}^s)_{uu}+({f_{{\rm n2}}}^s)_{vv}(v')^2)\cdot\nu=2f_{31}(0)$$
at the origin, where $'=d/du$.
These show the assertion.
\end{proof}

\subsection{Geometric invariants of deformation}
In \cite{hs}, the bias and the secondary cuspidal curvature are defined for non-degenerate singular point which is not a cuspidal edge.
The bias measures the bias of
a curve around the singular point, and the secondary cuspidal curvature
measures sharpness of $5/2$-cusp.
We calculate these invariants for deformations of cuspidal $S_1$ singularities.

\begin{definition}\cite[Definition 3.7]{hs}
Let $f:(\R^2,0)\to(\R^3,0)$ be a frontal satisfying $j^2f(0)=(u,v^2,0)$
such that $f$ at $v=0$ is not a cuspidal edge.
We take a coordinate system $(u,v)$ satisfying $S(f)=\{v=0\}$.
Then there exists a  vector field $\tilde \eta|_{\{v=0\}}$ such that
\begin{eqnarray}
\tilde\eta f(u,0)=0,\quad\langle f_u,\tilde \eta^2f\rangle(u,0)=\langle f_u,\tilde \eta^3f\rangle(u,0)=0.\label{null}
\end{eqnarray}
Then there exists $\ell$ such that $\tilde\eta^3f(0,0)=\ell\tilde\eta^2f(0,0)$.
Using this vector field $\tilde \eta$ and $\ell$ the
{\it bias} $r_b$ and the {\it secondary cuspidal curvature} $r_c$ are defined by
$$
\left. r_b=\dfrac{|f_u|^2\det(f_u,\tilde\eta^2f, \tilde \eta^4f)}{|f_u\times\tilde\eta^2f|^3}\right|_{(u,v)=(0,0)},
$$

$$
\left. r_c=\dfrac{|f_u|^{5/2}\det(f_u,\tilde\eta^2f,3\tilde\eta^5f-10\ell\tilde \eta^4f)}{|f_u\times\tilde\eta^2f|^{7/2}}\right|_{(u,v)=(0,0)}.
$$
\end{definition}

\begin{theorem}
Let $f:(\R^2,0)\to(\R^3,0)$ be a deformation of cuspidal $S_1$ singularities.
Then the bias $r_b$ and the secondary cuspidal curvature $r_c$ of $f$ at $(u(\pm\tilde s), 0, -\tilde s^2)$ can be expanded as  a function of $\tilde s$ as follows:
$$r_b=6c_{2}(0,0,0)+\dfrac{6(-2f_{21}(0)(c_{0})_u(0,0)+(c_{2})_u(0,0,0))}{d_{20}}\tilde s+O_{\tilde s}(2),$$
$$r_c=45\sqrt{2}c_{3}(0,0,0)+\dfrac{45\sqrt{2}(c_{3})_u(0,0,0)}{d_{20}}\tilde s+O_{\tilde s}(2).$$
\end{theorem}

\begin{proof}
Let us set $f={f_{{\rm n2}}}^s$.
 We set
  $$a_1=-\dfrac{f_{vv}\cdot f_u}{f_u\cdot f_u}\Big\vert_{v=0}, \quad a_2=-\dfrac{f_{vvv}\cdot f_u}{2f_u\cdot f_u}\Big\vert_{v=0}$$
  and
   $\tilde \eta=(va_1+v^2 a_2)\partial_u+\partial_v.$
   Then $\tilde \eta $ is a vector field and it satisfies $\tilde \eta f(u(\tilde s),0)=0$,
   and$\langle f_u,\tilde \eta^2f\rangle(u(\tilde s),0)=\langle f_u,\tilde \eta^3f\rangle(u(\tilde s),0)=0$.
   Thus this  $\tilde\eta$ satisfies \eqref{null}.
   Since
   $$\tilde \eta^2 f|_{v=0}=a_1f_u+f_{vv},\quad \tilde \eta^3 f|_{v=0}=f_{vvv}++2a_2f_u,$$
$$\tilde \eta^4 f|_{v=0}=f_{vvvv}+3a_1^2f_{uu}+6a_1f_{uvv}+3a_1(a_1)_uf_{uu},$$
$$\tilde \eta^5 f|_{v=0}=f_{vvvvv}+8(a_1)_ua_2f_u+20a_2f_{uvv}+20a_1a_2f_{uu}+12a_1(a_2)_uf_u+10a_1f_{uvvv},$$
it hold that $$\tilde \eta^3 f(u(\tilde s),0)=\ell \tilde \eta^2 f(u(\tilde s),0),$$
where
 $$\ell=-\dfrac{3\tilde s^4f_{24}(0,s)f_{34}(0,-\tilde s^2)}{1-\tilde s^4f_{24}(0,-\tilde s^2)c_{0}(0,-\tilde s^2)f_{34}(0,-\tilde s^2)+\tilde s^4f_{34}(0,-\tilde s^2)^2}.$$
Therefore,  we obtain the assertion at  $(u(\pm\tilde s), 0, -\tilde s^2)$.
\end{proof}
The coefficients $c_2(0,0,0)$ and $c_3(0,0,0)$ control the bias and the secondary cuspidal curvature
at cuspidal cross caps appearing a deformation,
and $(c_{0})_u(0,0)$, $(c_{2})_u(0,0,0)$ and $(c_{3})_u(0,0,0)$ control the changes of these invariants at cuspidal cross caps in the deformation.

\subsection{Geometry on trajectory of singular points}
We give a geometric meaning of
the lowest order coefficients $f_{24}(0, 0)$ and $f_{34}(0, 0)$
including the deformation parameters.
The trajectory of the singular points $S_2({f_{{\rm n2}}}^{-\tilde{s}^2})$ 
for the deformation of the cuspidal $S_1^\pm$ singularities $f={f_{{\rm n2}}}^s$ 
is a space curve passing through the origin. It is parameterized by
$$
\gamma (\tilde{s}):={f_{{\rm n2}}}^{-\tilde{s}^2}(u(\tilde{s}),0),
$$
where $u(\tilde{s})$ is given in Theorem \ref{thm:singcurve}.
Then the curvature $\kappa$ of $\gamma$ as a space curve at $\tilde s=0$ satisfies
$\kappa=2(f_{21}^2(0)+f_{31}^2(0))^{1/2}$.
Moreover, if $f_{21}^2(0)+f_{31}^2(0)\ne0$, then
$$f_{24}(0,0)=\dfrac{\kappa\tau f_{31}(0)-f_{21}(0)d_{20}\kappa'+3\kappa \dfrac{df_{31}}{du}(0)}{3\kappa d_{20}^2}$$
and
$$f_{34}(0,0)=-\dfrac{\kappa\tau f_{21}(0)+f_{31}(0)d_{20}\kappa'-3\kappa \dfrac{df_{31}}{du}(0)}{3\kappa d_{20}^2}$$
hold, where $\tau$ is the torsion of $\gamma$, and  
$\kappa'=d\kappa/d\tilde{s}$.
In fact, it hold that\\
$\kappa'=\dfrac{6f_{21}(0)d_{20}}{\sqrt{f_{21}(0)^2+f_{31}(0)^2}}f_{24}(0,0)+\dfrac{6f_{31}(0)d_{20}}{\sqrt{f_{21}(0)^2+f_{31}(0)^2}}f_{34}(0,0)$\\
\hspace{90mm}$+\dfrac{6(f_{21}(0)\dfrac{df_{21}}{du}(0)+f_{31}(0)\dfrac{df_{31}}{du}(0))}{d_{20}\sqrt{f_{21}(0)^2+f_{31}(0)^2}}$\\
and\\
$\tau=\dfrac{3f_{31}(0)d_{20}^2}{\sqrt{f_{21}(0)^2+f_{31}(0)^2}}f_{24}(0,0)+\dfrac{6f_{21}(0)d_{20}^2}{\sqrt{f_{21}(0)^2+f_{31}(0)^2}}f_{34}(0,0)$\\
\hspace{90mm}$+\dfrac{3(-f_{31}(0)\dfrac{df_{21}}{du}(0)+f_{21}(0)\dfrac{df_{31}}{du}(0))}{d_{20}\sqrt{f_{21}(0)^2+f_{31}(0)^2}}.$
\appendix
\section{Geodesic and normal curvatures of even functions}
\begin{proposition}\label{prop:kgkn}
Let a curve $\hat c:(\R,0)\to(\R^3,0)$ on a frontal in $\R^3$ satisfy that the $4$-jets of $\hat c_1(v), \hat c_2(v), \hat c_3(v)$ are even functions.
Here, we set $\hat c=(\hat c_1, \hat c_2, \hat c_3)$.
Let $\nu (v)$  a unit normal vector field of the frontal along $\hat c$.
Then the geodesic and normal curvatures at $v=0$ satisfy that
\begin{align*}
\kappa_g&=\dfrac{1}{3}\ep\dfrac{\det(\hat c_{vv},\hat c_{vvvv},\nu)}
{|\hat c_{vv}|^3}\Big\vert_{v=0},\\
\kappa_n&=\dfrac{1}{3}\dfrac{\hat c_{vvvv}\cdot\nu}
{|\hat c_{vv}|^2}\Big\vert_{v=0},
\end{align*}
where $\ep\in\{+1,-1\}$ is determined whether limit is taken from $v>0$ or $v<0$.
\end{proposition}
\begin{proof}
Since $\hat c$ is an even function, one can set
$$
\hat c_v(v)=va+v^3b(v)
$$
by a constant vector $a$ and $b(v)$.
By a direct calculation, 
$$
\hat c_{vv}(v)=a+3v^2b(v)+v^3b'(v),
$$
$$
\hat c_{vvvv}(v)=6b(v)+6vb'(v)+12v^2b'(v)+9v^2b''(v)+v^3b'''(v),
$$
 hold where $'=d/dv$.
Therefore the geodesic curvature is
$$
\kappa_g=\dfrac{\det(\hat c_v,\hat c_{vv},\nu)}{|c_v|^3}\Big\vert_{v=0}=\ep\dfrac{2\det(a,b(0),\nu)}{|c_1|^3}
$$
On the other hand, we calculate
$$
\dfrac{\det(\hat c_{vv},\hat c_{vvvv},\nu)}{|c_{vv}|^3}\Big\vert_{v=0}=\dfrac{6\det(a,b(0),\nu)}{|c_1|^3}.
$$
Thus we obtain the assertion for the geodesic curvature.

By $\hat c_v\cdot\nu\equiv0$, it hold that $(a+v^2b(v))\cdot\nu\equiv0$ for any $v$.
Thus by $\nu\cdot a=-v^2b(v)\cdot\nu$, we have
$$
\kappa_n=\dfrac{\hat c_{vv}\cdot\nu}{|c_v|^2}\Big\vert_{v=0}=\dfrac{2b(0)\cdot\nu}{|c_1|^2}.
$$
On the other hand,
it hold that
$$
\dfrac{\hat c_{vvvv}\cdot\nu}
{|\hat c_{vv}|^2}\Big\vert_{v=0}=\dfrac{6b(0)\cdot\nu}{|c_1|^2}.
$$
Thus we obtain the assertion of the normal curvature.

\end{proof}

\begin{acknowledgements}
The author would like to thank Kentaro Saji for fruitful discussion and comments.
\end{acknowledgements}

\medskip
{\footnotesize
\begin{flushright}
\begin{tabular}{l}

Department of Mathematics,\\
Graduate School of Science, \\
Kobe University, \\
Rokkodai 1-1, Nada, Kobe \\
657-8501, Japan\\
E-mail: {\tt 231s010s@stu.kobe-u.ac.jp}
\end{tabular}
\end{flushright}

\end{document}